\documentclass{amsart}
\usepackage{amsmath,amssymb,amsfonts,enumerate}
\usepackage{graphicx}
\newcommand{\R}{\mathbb{R}}
\newcommand{\A}{\mathsf{A}}
\newcommand{\threebar}[1]{{\left\vert\kern-0.25ex\left\vert\kern-0.25ex\left\vert #1 
    \right\vert\kern-0.25ex\right\vert\kern-0.25ex\right\vert}}

\newcommand{\B}{\mathsf{B}}

\newtheorem{question}{Question}
\newtheorem{lemma}{Lemma}
\newtheorem{proposition}[lemma]{Proposition}

\newtheorem{theorem}{Theorem}
\DeclareMathOperator{\tr}{tr}

\begin{document}
\title[Periodic stability does not imply global stability]{Stability with respect to periodic switching laws does not imply global stability under arbitrary switching}
\author{Ian D. Morris}
\maketitle

\begin{abstract}
R. Shorten, F. Wirth, O. Mason, K. Wulff and C. King have asked whether a linear switched system is guaranteed to be globally uniformly stable under arbitrary switching if it is known that every trajectory induced by a periodic switching law converges exponentially to the origin.  Positive answers to this question have previously been announced for linear switched systems of order two and three. We answer this question negatively in all higher orders by constructing a fourth-order linear switched system with four switching states which is not uniformly exponentially stable but which has the property that every trajectory defined by a periodic switching law converges exponentially to the origin. 
We argue informally that positive linear systems with this combination of properties are likely to exist in sufficiently high dimensions.

\bigskip

\end{abstract}

\section{Introduction}


A linear switched system is a dynamical system which consists of a family of linear subsystems together with a \emph{switching law} whose role is to determine which of the linear subsystems will govern the system's behaviour at each future time. Such systems have been studied extensively in the last few decades based on their relevance to phenomena including mechanical systems, power systems, traffic control and others (for references see for example \cite{LiAn09}) and are discussed in books and surveys including \cite{Li03,ChMaSi25,ShWiMaWuKi07,LiAn09,SuGe11}. This note concerns the problem of finding sufficient conditions for the stability of the origin for a general linear switched system. 

If $\A=\{A_0,\ldots,A_{N-1}\}$ is a set of real $d \times d$ matrices then we will say that a trajectory of the linear switched system defined by $\A$ is any solution $x(t)$ to a non-autonomous ordinary differential equation of the form
\begin{equation}\label{eq:dimple}\dot{x}(t) = A(t)x(t)\end{equation}
where $A(t)$, which we call the \emph{switching law}, is a function taking values in $\A$. The matrices $A_0,\ldots,A_{N-1}$ themselves will be referred to as \emph{switching states}. If more generally $A(t)$ is allowed to take values in the convex hull of the set of switching states then $x(t)$ is called a trajectory of the \emph{relaxed linear switched system} defined by $\A$. This note is concerned  primarily with  the relaxed version of the equation \eqref{eq:dimple} in the context of an arbitrary switching law $A(t)$, but has consequences for the non-relaxed system in the case of piecewise constant switching laws. 

The linear switched system defined by $\A$ is conventionally called \emph{globally uniformly asymptotically stable} or GUAS if its trajectories converge uniformly to the origin at a uniform rate in the following sense: there exists a class $\mathcal{KL}$ function $\beta \colon [0,\infty) \times [0,\infty) \to [0,\infty)$ such that $\|x(t)\|\leq \beta(\|x(0)\|,t)$ for all $t\geq 0$, for every trajectory $x(t)$, and for all $\A$-valued switching laws $A(t)$. For linear switched systems the function $\beta$ can without loss of generality be taken to be an exponential function $\beta(\|x(0)\|,t)\equiv Ke^{-\kappa t}$ --  see for example \cite{Li03,ChMaSi25} for a proof -- and so in our context GUAS implies uniform exponential stability of the origin.  The linear switched system defined by $\A$ is also conventionally called \emph{Lyapunov stable} if there exists $K>0$ depending only on $\A$ such that $\|x(t)\| \leq K \|x(0)\|$ holds for all $t \geq0$, for all trajectories $x(t)$ and for all $\A$-valued switching laws $A(t)$. It is known that the system \eqref{eq:dimple} is GUAS if and only if the corresponding relaxed linear switched system is GUAS, and by considering the case of constant switching laws this implies in particular that if $\A$ is GUAS then every matrix in the convex hull of $\A$ must be Hurwitz: see for example \cite[Remark 1.19]{ChMaSi25}. In what follows we will by default consider the relaxed version of the system \eqref{eq:dimple}.


It is a problem of fundamental interest to find general conditions which guarantee that a given system is GUAS. It is necessary for GUAS that every element of the convex hull of $\A$ should be a Hurwitz matrix -- or equivalently, that every constant switching law $A(t)$ with values in the convex hull should give rise only to exponentially stable trajectories -- but it is also known that this condition does not itself imply GUAS. More precisely, certain linear switched systems in two dimensions have the property that every constant switching law gives rise only to exponentially stable trajectories, yet certain \emph{periodic} switching laws give rise to unstable trajectories. Examples of such systems have been investigated for example in  \cite{BaBoMa09,Bo02,MaLa03,Mu23,PrMu24,PyRa96}, where second-order systems were studied, and higher-dimensional systems with similar properties have been investigated in \cite{BoMo26,ChGaMa15,FaMaCh09,GuShMa07}. Notably, in the second-order case it is possible to demonstrate that a system is stable if and only if every constant switching law and every periodic switching law gives rise only to exponentially decaying trajectories. The stability analysis of second-order switched systems under arbitrary switching is thus reduced to the much simpler problem of testing whether or not the system is stable for a certain explicit class of  ``most unstable switching laws'' all of which are either constant, or periodic with a simple and explicit structure. 

It is natural to ask whether the same result holds for higher order systems. The following question was posed by R. Shorten \emph{et al.} in the survey article \cite{ShWiMaWuKi07}:

\begin{question}\label{qu:that-question}
Let $A_0, \ldots, A_{N-1}$ be $d\times d$ matrices. Suppose that every matrix in the convex hull of $\{A_0,\ldots,A_{N-1}\}$ is Hurwitz, and suppose additionally that every trajectory of \eqref{eq:dimple} which arises from a periodic switching law converges to the origin. Does it follow that the linear switched system defined by $\{A_0,\dots,A_{N-1}\}$ is GUAS?
\end{question}

In dimension two a positive answer to this question follows from works such as \cite{BaBoMa09,Bo02,MaLa03,Mu23,PrMu24,PyRa96} as mentioned above. E.S. Pyatnitskiy and L.B. Rapoport also announced a positive answer to this question in the three-dimensional case (see \cite[\S{X}]{PyRa96}) but did not give a detailed proof. A complete, positive resolution of Question \ref{qu:that-question} for three-dimensional systems was recently announced by the author and J. Bochi in the separate preprint \cite{BoMo26}. 
The purpose of this note is to prove that the answer to Question \ref{qu:that-question} is \emph{negative} in the case $N=d=4$, which trivially implies a negative result for all higher dimensions $d$. To do this we will construct a linear switched system whose most unstable switching laws are \emph{quasi-periodic}, being in a certain sense the product of two periodic motions which resonate with one another at an irrational frequency.

A related question addressed by L. Gurvits, R. Shorten and O. Mason \cite{GuShMa07}, also posed independently by D. Angeli, asked whether every \emph{positive} linear switched system $\{A_0,\ldots,A_{N-1}\}$ with the property that every matrix in the convex hull of $\{A_0,\ldots,A_{N-1}\}$ is Hurwitz must be GUAS. Here the system is called \emph{positive} if all of the matrices $A_0,\ldots,A_{N-1}$ are Metzler matrices, or equivalently if for every trajectory $x(t)$ which begins at a non-negative vector $x(0)$, the vector $x(t)$ is non-negative for all $t \geq 0$. This question was resolved negatively by Gurvits \emph{et al} in \cite{GuShMa07} in the case where the dimension is sufficiently large, and a counterexample in dimension $3$ was later obtained by Fainshil, Margaliot and Chigansky in \cite{FaMaCh09}. In view of this question it is natural to ask whether counterexamples to Question \ref{qu:that-question} may be constructed which have the property of being positive. While we are not able to resolve this question definitively, we believe that such counterexamples should exist. This matter will be discussed further in the conclusions.

The result which we demonstrate in this note is the following negative answer to Question \ref{qu:that-question}:
\begin{theorem}\label{th:headline}
There exists a set $\A$ of four $4\times 4$ real matrices with the property that every solution of \eqref{eq:dimple} corresponding to a periodic Lebesgue measurable switching law $A(t)$ converges exponentially to zero, but such that the linear switched system defined by $\A$ is not GUAS.
\end{theorem}
Theorem \ref{th:headline} will be stated in a more precise form in the following section, but admits a relatively simple informal description as follows. As mentioned earlier, it has been shown that there exist pairs of $2\times 2$ real matrices $A_0, A_1$ which admit periodic most-unstable switching laws. By extending the arguments used in the works mentioned earlier, one may show additionally that for these pairs of matrices, for each initial vector there exists a \emph{unique} periodic switching law giving rise to a periodic solution of \eqref{eq:dimple}. The period of this switching law is moreover independent of the initial state vector. Call this period  $T$, say. Combining two copies of such a system using a tensor product construction, one may define a system of four $4\times 4$ matrices $B_0, B_1, B_2, B_3$ whose trajectories can be reduced to tensor products of pairs of independent trajectories of the aforementioned two-dimensional factor system $A_0, A_1$, with the second of the two independent trajectories moreover being rescaled in time by an irrational factor $\alpha$. It is then impossible for any periodic switching law to simultaneously result in non-decaying trajectories of both of the two-dimensional  factor systems, because in order to yield a non-decaying trajectory of the first  factor system its period must be an integer multiple of $T$, but in order to also have the same effect on the second  factor system its period must also be an integer multiple of $T/\alpha$, which is impossible by the irrationality of $\alpha$. On the other hand by a suitable \emph{quasi-periodic} control input one may cause both factor systems to repeat themselves periodically on rationally independent periods, resulting in a non-decaying trajectory and proving that the system is not GUAS. We will render this sketch argument into an explicit and rigorous form in Theorems \ref{th:example} and \ref{th:not-periodic} below.

\section{Mathematical background and statements of theorems}

Henceforth the switching law $A(t)$ will be assumed to be a function from the non-negative real line $[0,\infty)$ to the convex hull of the set of matrices $\{A_0,\ldots,A_{N-1}\}$ which is \emph{Lebesgue measurable}, which is the most general hypothesis under which the equation \eqref{eq:dimple}  makes sense. In this context the Carath\'eodory existence theorem (see for example \cite[Appendix C]{So98}) implies that for every switching law $A(t)$ and initial condition $x(0)$ there exists a unique absolutely continuous function $x(t)$ which satisfies \eqref{eq:dimple} for Lebesgue a.e. $t \geq 0$ and has initial state $x(0)$, and solutions to equations such as \eqref{eq:dimple} will always be understood in this sense. While the hypothesis that $A(t)$ is Lebesgue measurable is more convenient when giving proofs, the conclusions of our results will remain valid if it is always assumed instead that $A(t)$ is a piecewise constant function taking values in $\{A_0,\ldots,A_{N-1}\}$. In the particular case $N=2$ one may of course write $A(t)\equiv (1-u(t))A_0 + u(t)A_1$ where $u \colon [0,\infty) \to [0,1]$ is measurable, and in the case $N=2$ we will prefer to write the switching law $A(t)$ in this form. Abusing terminology very slightly, we will refer to the function $u$ as the switching law in this context. Throughout the rest of this work we use the notation $M_d(\R)$ to denote the set of all $d\times d$ real matrices.

Before stating our main result we require the following theorem which captures some essential properties of the two-dimensional case:
\begin{theorem}\label{th:example}
There exists a pair of matrices $A_0, A_1 \in M_2(\R)$ with the following properties:
\begin{enumerate}[(i)]
\item\label{it:hurwitz}
The matrices $I$, $A_0$ and $A_1$ are linearly independent, and for every $\gamma \in [0,1]$ the matrix $(1-\gamma)A_0 + \gamma A_1$ is Hurwitz.
\item\label{it:lyapunov}
There exists $C>0$ such that for every Lebesgue measurable function $u \colon [0,\infty) \to [0,1]$, every absolutely continuous solution $x(t)$ to the differential equation
\begin{equation}\label{eq:that-equation}\dot{x}(t)=\left((1-u(t))A_0 + u(t)A_1\right)x(t),\end{equation}
satisfies $\|x(t)\|\leq C\|x(0)\|$ for all $t \geq 0$.
\item\label{it:unique}
For every nonzero $w \in \R^2$ there exists a Lebesgue measurable function $u \colon [0,\infty) \to [0,1]$ such that the solution to \eqref{eq:that-equation} with initial condition $x(0)=w$ 
is periodic. The function $u$ is unique up to Lebesgue measure zero, is periodic with period $T>0$ not depending on $w$, is piecewise constant, and admits the following description for a certain constant $T_0 \in (0,T)$. Let  $u_0 \colon [0,\infty) \to [0,1]$ be the function which takes the value $0$ on $[0,T_0)$, then takes the value $1$ on the following interval $[T_0,T)$, and repeats periodically with period $T$. Then $u(t)=u_0(t+t_0)$ a.e. for some real number $t_0 \in [0,T)$ which depends only on $w$.
\item\label{it:exp}
If $u \colon [0,\infty) \to [0,1]$ is a periodic Lebesgue measurable function then either every solution to \eqref{eq:that-equation} converges exponentially to the origin, or there exists a non-constant periodic solution to \eqref{eq:that-equation}. In the latter case we have $u(t)=u_0(t+t_0)$ a.e. for some real number $t_0 \in [0,T)$ and in particular $u$ has period $T$. 
\end{enumerate}
\end{theorem}
Theorem \ref{th:example} is a modest extension of now-standard stability analyses of second-order bilinear systems under arbitrary switching. Such systems have been analysed independently on several occasions and their properties are treated in various works including \cite{AgMo23,Bo02,BaBoMa09,ChMaSi25,MaLa03,PrMu24,PyRa96,YaXiLe12}.
The results in these works imply the existence of examples satisfying \eqref{it:hurwitz}--\eqref{it:lyapunov} and the first sentence of \eqref{it:unique}, but in these works the uniqueness clause of \eqref{it:unique} and its consequences as described in \eqref{it:exp} are not explicitly addressed. In this note the uniqueness property, and the exponential convergence to zero of trajectories generated by all other periodic switching laws, will play central roles. For this reason we include a proof of Theorem \ref{th:example} in an appendix.

In this work we will also require some basic properties of the Kronecker tensor product of matrices and vectors, which we now describe. Proofs of the following assertions may be found in \cite{HoJo94}. Recall that if $A$ and $B$ are rectangular matrices of dimensions $n_1\times m_1$ and $n_2 \times m_2$ respectively,
\[A=\begin{pmatrix} a_{11}& a_{12} & \cdots &a_{1m_1}\\
a_{21}&a_{22}&\cdots &a_{2m_1}\\
\vdots &\vdots &\ddots &\vdots \\
a_{n_11} &a_{n_12}&\cdots & a_{n_1m_1}\end{pmatrix},
 \]
 \[B=\begin{pmatrix} b_{11}& b_{12} & \cdots &b_{1m_2}\\
b_{21}&b_{22}&\cdots &b_{2m_2}\\
\vdots &\vdots &\ddots &\vdots \\
b_{n_21} &b_{n_22}&\cdots & b_{n_2m_2}\end{pmatrix},\]
then their Kronecker product is the $n_1n_2 \times m_1m_2$ matrix $A\otimes B$ defined by
\[A\otimes B =\begin{pmatrix} a_{11}B & a_{12} B& \cdots &a_{1m_1}B\\
a_{21}B&a_{22}B&\cdots &a_{2m_1}B\\
\vdots &\vdots &\ddots &\vdots \\
a_{n_11}B &a_{n_12}B&\cdots & a_{n_1m_1}B\end{pmatrix}.\]
That is, $A\otimes B$ is the $n_1n_2 \times m_1m_2$ matrix formed by placing all of the possible $n_2\times m_2$ matrices of the form $a_{ij} B$ in an $n_1 \times m_2$ grid in the pattern of the entries of $A$. This construction may also be applied to row and column vectors by considering them as $1\times n$ or $n\times 1$ matrices respectively. In particular if $u$ and $v$ are column vectors of dimension $n$ and $m$ respectively then $u \otimes v$ is a column vector of dimension $nm$. The Kronecker product respects multiplication:
\[(A_0 \otimes A_1) (B_0 \otimes B_1)=(A_0B_0)\otimes (A_1B_1)\]
and also respects addition and scalar multiplication with respect to either (but not both) of the matrices $A$ and $B$, that is
\[(\lambda A_0+A_1) \otimes B = \lambda (A_0\otimes B) + A_1 \otimes B,\]
\[A \otimes (\lambda B_0+B_1) = \lambda (A\otimes B_0) + A \otimes B_1,\]
and using these rules it is not difficult to demonstrate that if $A_1$ and $A_2$ are square matrices then
\[\exp(A_0 \otimes I_1 + I_0 \otimes A_1) = (\exp A_0)\otimes (\exp A_1)\]
where $I_0$ and $I_1$ are identity matrices of the same dimension as $A_0$ and $A_1$ respectively. If $u(t)$ and $v(t)$ are differentiable vector-valued functions then it is easily shown using the above properties that $u(t)\otimes v(t)$ is also differentiable and has derivative $\dot{u}(t)\otimes v(t)+u(t)\otimes \dot{v}(t)$. Lastly, if $A$ and $B$ are square matrices of any dimension, and $u$ and $v$ are column vectors of any dimension, then
\[\|A\otimes B\| = \|A\|\cdot \|B\|\]
and
\[\|u \otimes v\|= \|u\|\cdot\|v\|\]
where $\|X\|$ denotes the Euclidean operator norm of the matrix $X$ and where $\|w\|$ denotes the Euclidean norm of the vector $w$. We will use these properties extensively in the following section without further comment.

We are now in a position to state our main result, which is the following more precise formulation of Theorem \ref{th:headline}:

\begin{theorem}\label{th:not-periodic}
Let $A_0, A_1 \in M_2(\R)$ have the properties described in Theorem \ref{th:example}. Let $\alpha>0$ be irrational, and define four matrices $B_0, B_1, B_2, B_3 \in M_4(\R)$ by 
\[\begin{array}{c} B_0:= A_0\otimes I + \alpha \cdot I \otimes A_0,\\
B_1:= A_1 \otimes I + \alpha \cdot I \otimes A_0,\\
B_2:= A_0\otimes I + \alpha \cdot I \otimes A_1,\\
B_3:= A_1 \otimes I + \alpha\cdot I \otimes A_1,\end{array}\]
where $I$ denotes the $2 \times 2$ identity matrix. Let $\B$ denote the convex hull of $\{B_0,\ldots,B_3\}$. Then:
\begin{enumerate}[(i)]
\item\label{it:you}
There exists a constant $K>0$ such that for every Lebesgue measurable switching law $B \colon [0,\infty) \to \B$ and every $w \in \R^4$ the solution to
\begin{equation}\label{eq:himb-an-equation}
\dot{x}(t)=B(t)x(t)\end{equation}
\begin{equation}\label{eq:himb-a-other-equation} x(0)=w\end{equation}
satisfies $\|x(t)\|\leq K\|x(0)\|$ for all $t\geq 0$.
\item\label{it:themb}
If $B \colon [0,\infty) \to \B$ is Lebesgue measurable and periodic then every solution to \eqref{eq:himb-an-equation} converges exponentially to the origin.
\item\label{it:himb}
There exist a piecewise constant switching law $B \colon [0,\infty) \to \{B_0, B_1, B_2, B_3\}$ and vector $w \in \R^4$ such that the solution to \eqref{eq:himb-an-equation}, \eqref{eq:himb-a-other-equation} does not accumulate at the origin.
\end{enumerate}
\end{theorem}
The three clauses of Theorem \ref{th:not-periodic} in turn imply that the linear switched system defined by $\B$ is Lyapunov stable, exponentially stable with respect to every periodic switching law, and not globally asymptotically stable with respect to piecewise constant switching signals (hence also with respect to arbitrary switching).

The switching law $B(t)$ constructed in Theorem \ref{th:not-periodic}\eqref{it:himb} is in fact \emph{quasi-periodic} in the following sense: if $u_0$ is as described in Theorem \ref{th:example} then $B(t)=B_0$ when $u_0(t)=u_0(\alpha t)=0$, $B(t)=B_1$ when $1-u_0(t)=u_0(\alpha t)=0$, $B(t)=B_2$ when $u_0(t)=1-u_0(\alpha t)=0$, and  $B(t)=B_3$ when $1-u_0(t)=1-u_0(\alpha t)=0$. It is possible to show that pairs of successive switching times of this law $B(t)$ can occur arbitrarily close together, but that three switching times cannot occur arbitrarily close together.

\section{Proof of Theorem \ref{th:not-periodic}}
Throughout this section we fix two matrices $A_0, A_1 \in M_2(\R)$ having the properties described in Theorem \ref{th:example}, and let $\alpha>0$ and $B_0, B_1, B_2, B_3$ be as in the statement of Theorem \ref{th:not-periodic}. We begin by establishing a result which will allow us to rewrite switching laws with values in $\B$  in terms more directly compatible with the statement of Theorem \ref{th:example}:

\begin{proposition}\label{pr:tensorbollox}
There exist two linear maps $\ell_0, \ell_1 \colon M_4(\R) \to \R$ such that for every $B \in \B$ we have
\[B-B_0=\ell_0(B-B_0)\cdot (B_1-B_0) + \ell_1(B-B_0)\cdot (B_2-B_0).\]
and additionally $0\leq \ell_0(B-B_0), \ell_1(B-B_0)\leq 1$.
\end{proposition}

To prove the proposition we require two lemmas.

\begin{lemma}\label{le:tensorbollocks}
Let $C$ and $D$ be $2 \times 2$ real matrices. If $C \otimes I=I\otimes D$ then $C=D=\gamma I$ for some real number $\gamma$.
\end{lemma}
\begin{proof}
If $C \otimes I=I\otimes D$ with
\[C=\begin{pmatrix} c_{11} & c_{12}\\ c_{21} & c_{22}\end{pmatrix},\quad D=\begin{pmatrix} d_{11} & d_{12}\\ d_{21} & d_{22}\end{pmatrix}\]
then the equation $C \otimes I = I \otimes D$ becomes
\[\begin{pmatrix} c_{11}&0 &c_{12}&0\\
0&c_{11} &0&c_{12}\\
c_{21}&0&c_{22}&0\\
0&c_{21}&0&c_{22}\end{pmatrix}=\begin{pmatrix} d_{11}&d_{12} &0&0\\
d_{21}&d_{22} &0&0\\
0&0&d_{11}&d_{12}\\
0&0&d_{21}&d_{22}\end{pmatrix}
\]
and we immediately find that $c_{12}=c_{21}=d_{12}=d_{21}=0$ and $c_{11}=d_{11}=c_{22}=d_{22}$ as required.
\end{proof}

\begin{lemma}
The matrices $B_1-B_0$ and $B_2-B_0$ are linearly independent.
\end{lemma}
\begin{proof}
We have
\begin{align*}B_1-B_0&=(A_1-A_0)\otimes I,\\
B_2-B_0&=\alpha \cdot I\otimes(A_1-A_0),\end{align*}
and if these two matrices are not linearly independent then an equation of the form
\[(A_1-A_0)\otimes I = \gamma \cdot I \otimes (A_1-A_0)\]
must hold. By Lemma \ref{le:tensorbollocks} this is only possible if $A_1-A_0$ is a scalar multiple of the identity matrix, which contradicts Theorem \ref{th:example}\eqref{it:hurwitz}.
\end{proof}

We now prove Proposition \ref{pr:tensorbollox}:

\begin{proof}
Since $B_1-B_0$ and $B_2-B_0$ are linearly independent, there exist linear maps $\ell_0, \ell_1 \colon M_4(\R) \to \R$ such that
\[\ell_0(B_1-B_0)=1,\quad \ell_0(B_2-B_0)=0,\]
\[\ell_1(B_1-B_0)=0, \quad \ell_1(B_2-B_0)=1.\]
If $B$ belongs to the convex hull of $B_0, B_1, B_2, B_3$ then we may write
\[B=\sum_{i=0}^3 \beta_i B_i\]
where $\sum_{i=0}^3 \beta_i=1$ and where $0 \leq \beta_i \leq 1$ for every $i$. The identity
\[B_0+B_3=B_1+B_2\]
follows from the definition of the matrices $B_i$, so
\begin{align*}B-B_0 &= (\beta_0-1)B_0 + \beta_1B_1+\beta_2B_2 + \beta_3 B_3\\
&= (\beta_0-1)B_0 + \beta_1B_1+\beta_2B_2 + \beta_3(B_1+B_2-B_0)\\
&=(\beta_1+\beta_3)(B_1-B_0)+(\beta_2+\beta_3)(B_2-B_0).\end{align*}
Applying $\ell_0$ and $\ell_1$ to both sides of the above equation immediately yields
\[\ell_0(B-B_0)=\beta_1+\beta_3,\]
\[\ell_1(B-B_0)=\beta_2+\beta_3.\] In particular
\[B-B_0=\ell_0(B-B_0)(B_1-B_0)+\ell_1(B-B_0)(B_2-B_0)\]
and $0\leq \ell_i(B-B_0) \leq 1$ for $i=0,1$ as required.
\end{proof}

We now prove Theorem \ref{th:not-periodic}. For the remainder of this section we fix linear maps $\ell_0, \ell_1 \colon M_4(\R) \to \R$ as given by Proposition \ref{pr:tensorbollox}, and let $e_0, e_1$ denote the standard basis for $\R^2$. We observe that $e_0\otimes e_0$, $e_0\otimes e_1$, $e_1\otimes e_0$, $e_1\otimes e_1$ is precisely the standard basis for $\R^4$. Fix also $C,T>0$ as given by clauses \eqref{it:lyapunov} and \eqref{it:unique} of Theorem \ref{th:example}.

We begin by demonstrating clauses \eqref{it:you} and \eqref{it:themb} of Theorem \ref{th:not-periodic}. Let $B \colon [0,\infty) \to \mathsf{B}$ be a measurable switching law. We will show that if $w_0, w_1 \in \R^2$ are unit vectors then the solution to $\dot{y}(t)=B(t)y(t)$ with initial condition $y(0)=w_0\otimes w_1$ satisfies $\|y(t)\|\leq C^2$ for all $t\geq 0$, and if additionally $B$ is periodic then $y(t)$ converges to zero exponentially. Since the vectors $e_i\otimes e_j$ form a basis for $\R^4$ this implies (by linear superposition of solutions) that the same conclusions hold for a trajectory beginning at an arbitrary unit vector, except that an upper bound of $4C^2$ is obtained in place of $C^2$. In particular this special case directly implies the full statements of \eqref{it:you} and \eqref{it:themb} by appeal to linear superposition of solutions.

Fix unit vectors $w_0,w_1 \in \R^2$, therefore, and define switching laws $v_0, v_1 \colon [0,\infty) \to [0,1]$ by $v_0(t):=\ell_0(B(t)-B_0)$ and $v_1(t):=\ell_1(B(t)-B_0)$ respectively. Since $\ell_0$ and $\ell_1$ are continuous, $v_0$ and $v_1$ are measurable. Using Proposition \ref{pr:tensorbollox}, for every $t \geq 0$ we have
\begin{align*}B(t)&= B_0 + v_0(t)(B_1-B_0) + v_1(t)(B_2-B_0)\\
&=((1-v_0(t))A_0 + v_0(t)A_1)\otimes I \\
&\qquad + \alpha\cdot I \otimes ((1-v_1(t))A_0+v_1(t)A_1).\end{align*}
Define $x_0,x_1\colon [0,\infty) \to \R^2$ to be the unique absolutely continuous solutions to the initial value problems
\[\dot{x}_0(t) = ((1-v_0(t))A_0 + v_0(t)A_1) x_0(t),\] 
\[\dot{x}_1(t) = \alpha ((1-v_1(t))A_0 + v_1(t)A_1) x_1(t),\]
\[x_0(0)=w_0,\quad x_1(0)=w_1\]
and observe that $y(t):=x_0(t)\otimes x_1(t)$ is the unique absolutely continuous solution to $\dot{y}(t)=B(t)y(t)$ a.e. and $y(0)=w_0\otimes w_1$. Since $z(t):=x_1(\alpha t)$ solves
\[\dot{z}(t)=((1-v_1(\alpha t))A_0+v_1(\alpha t)A_1)z(t)\]
it follows from Theorem \ref{th:example}\eqref{it:lyapunov} that $\|x_1(t)\|=\|z(t/\alpha)\|\leq C$ for all $t \geq0$, and directly $\|x_0(t)\|\leq C$ for all $t \geq 0$. In particular $\|y(t)\|\leq C^2$ for all $t \geq 0$ as required.

If additionally $B(t)$ is periodic with period $p$ then each of $v_0(t)$ and $v_1(t)$ either is constant or is periodic with least period dividing $p$. If $v_i$ is constant a.e. then $x_i(t)\to0$ exponentially since every convex combination of $A_0$ and $A_1$ is Hurwitz, so suppose instead that both are periodic. Then $v_0(t)$ and $v_1(\alpha t)$ are both periodic but the ratio of their least periods is a rational multiple of $\alpha$, hence is irrational, so at least one of them is \emph{not} periodic with period $T$. It follows by Theorem \ref{th:example}\eqref{it:exp} that either $x_0(t)$ or $z(t)=x_1(\alpha t)$ converges exponentially to zero while the other is uniformly bounded by Theorem \ref{th:example}\eqref{it:lyapunov}. In particular $y(t)$ converges exponentially to zero as required. This completes the proof of \eqref{it:you} and \eqref{it:themb} for trajectories whose initial condition has the form $w_0\otimes w_1$ where $w_0,w_1 \in \R^2$ are unit vectors, and the general case follows by linear superposition as noted previously.

It remains to prove Theorem \ref{th:not-periodic}\eqref{it:himb}. By Theorem \ref{th:example}\eqref{it:unique} there exist $u \colon [0,\infty) \to [0,1]$ and nonzero $w \in \R^2$ such that the solution to
 \[\dot{x}(t)=((1-u(t))A_0 + u(t)A_1) x(t),\quad x(0)=w\]
 is periodic. Define $B \colon [0,\infty) \to \B$ by
 \begin{align*}
 B(t)&=(1-u(t))(1-u(\alpha t))B_0 + u(t)(1-u(\alpha t))B_1\\
 & \qquad + (1-u(t))u(\alpha t)B_2+u(t)u(\alpha t)B_3\\
 &=B_0 + u(t)(B_1-B_0)+u(\alpha t)(B_2-B_0)\\
 &=((1-u(t))A_0 + u(t)A_1)\otimes I \\
 &\qquad + \alpha\cdot I \otimes ((1-u(\alpha t))A_0+u(\alpha t)A_1)\end{align*}
 so that $y(t):=x(t)\otimes x(\alpha t)$ solves $\dot{y}(t)=B(t)y(t)$ a.e. and $y(0)=w\otimes w$. Since $x(t)$ is periodic and never zero, we have
 \[\inf_{t \geq 0} \|y(t)\| = \inf_{t \geq 0}\|x(t)\otimes x(\alpha t)\|  \geq \left(\inf_{t \geq 0} \|x(t)\|\right)^2>0\]
 so that $y(t)$ remains a bounded distance away from the origin at all times. The proof of the theorem is complete.
 
 \section{An explicit example}\label{se:explicit}
 We illustrate Theorem \ref{th:not-periodic} with an explicit example.
 Consider two matrices $A_0, A_1 \in M_2(\R)$ defined by
\[A_0=\begin{pmatrix} -1 & \frac{\sqrt{\tau}(\tau-1)}{\sqrt{2}}  \\ \frac{1-\tau}{\sqrt{2\tau}} & -\tau\end{pmatrix},\quad A_1=\begin{pmatrix} -\tau & \frac{\tau-1}{\sqrt{2\tau}}  \\ \frac{\sqrt{\tau}(1-\tau)}{\sqrt{2}} & -1\end{pmatrix},\]
where   $\tau=0.1299992\ldots $ is the unique real number which satisfies
\[\tau=\exp\left(\frac{\pi(\tau+1)}{2(\tau-1)}\right),\qquad 0<\tau<1.\]
This pair $A_0,A_1$ can be shown to satisfy the properties described in Theorem \ref{th:example}. We briefly outline a demonstration of this fact which applies the ideas of \cite{MaLa03} to exhibit a non-strict common Lyapunov function. Consider the function $f\colon \R^2 \to \R$ which is zero at the origin and which is otherwise given by
\[\left(x_1^2+\sqrt{2\tau} x_1x_2+\tau x_2^2\right)e^{\frac{2(\tau+1)}{\tau-1} \arctan \left(\frac{ x_1}{x_1+\sqrt{2\tau} x_2}\right)}\]
when $(x_1,x_2) \in \R^2$ satisfies $x_1x_2 \geq 0$ and by 
\[\left(\tau x_1^2-\sqrt{2 \tau }x_1x_2+x_2^2\right)e^{\frac{2(1+\tau)}{1-\tau} \arctan \left(\frac{x_2}{\sqrt{2 \tau}x_1-x_2}\right)}\]
when $x_1x_2\leq 0$. A prolonged but elementary calculation shows that $f$ 
is a $C^1$ function, is non-negative for all nonzero $(x_1,x_2)$, and satisfies
\begin{figure}\centering
\includegraphics[width=0.9\linewidth]{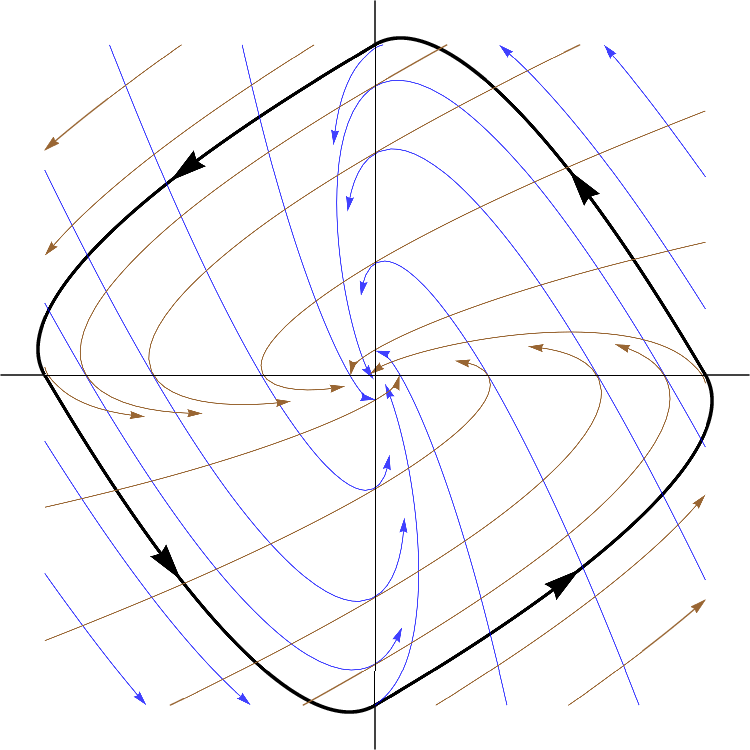}
\caption{The black line shows a periodic trajectory of the linear switched system defined by the matrices $A_0$, $A_1$ considered in section \ref{se:explicit}. Blue and brown flow lines follow the vector fields defined by $A_0$ and $A_1$ respectively. The periodic trajectory describes a level curve of the non-strict Lyapunov function $f$, and switches between the two vector fields upon crossing the horizontal or vertical axis.  }\label{fi:gure}
\end{figure}
\[\begin{array}{cl}
(\nabla f)(x_1,x_2) \cdot A_0(x_1,x_2)^T=0&\text{if }x_1x_2\geq 0,\\
(\nabla f)(x_1,x_2) \cdot A_0(x_1,x_2)^T<0&\text{if }x_1x_2<0\end{array}\]
and
\[\begin{array}{cl}(\nabla f)(x_1,x_2) \cdot A_1(x_1,x_2)^T=0&\text{if }x_1x_2\leq 0,\\
(\nabla f)(x_1,x_2) \cdot A_1(x_1,x_2)^T<0&\text{if }x_1x_2> 0.\end{array}
\]
These properties imply that  $f$ is a $C^1$ non-strict Lyapunov function for the corresponding switched system: if $x(t)$ is a nonzero solution to an equation
\[\dot{x}(t)=((1-u(t))A_0 + u(t)A_1)x(t)\]
then $f(x(t))$ is non-increasing, and is constant if and only if  the underlying switching law $u(t)$ satisfies $u(t)=0$ for a.e. $t$ such that the two co-ordinates of $x(t)$ have the same sign, and also satisfies $u(t)=1$ for a.e. $t$ such that the co-ordinates of $x(t)$ have different signs. A periodic trajectory of this system is shown in Figure \ref{fi:gure}. This pair of matrices has the property that $A_0$ is conjugated to $A_1$ by a rotation of $\pi/2$, and consequently the periodic trajectory is symmetrical with respect to this rotation. As a consequence of this symmetry, the switching laws corresponding to periodic trajectories consist of two bang intervals of equal duration.
Following the construction of Theorem \ref{th:not-periodic} with $\alpha:=\sqrt{2}$, the switched linear system defined by the four matrices
\[B_0=\begin{pmatrix}
-1-\sqrt{2}&\sqrt{\tau}(\tau-1)&\frac{\sqrt{\tau}(\tau-1)}{\sqrt{2}}&0\\
\frac{1-\tau}{\sqrt{\tau}}&-1-\tau\sqrt{2}&0&\frac{\sqrt{\tau}(\tau-1)}{\sqrt{2}}\\
\frac{1-\tau}{\sqrt{2\tau}}&0&-\tau-\sqrt{2}&\sqrt{\tau}(\tau-1)\\
0&\frac{1-\tau}{\sqrt{2\tau}}&\frac{1-\tau}{\sqrt{\tau}}&-\tau-\tau\sqrt{2}\\
\end{pmatrix},\]
\[B_1=\begin{pmatrix}-\tau-\sqrt{2}&\sqrt{\tau}(\tau-1)&\frac{\tau-1}{\sqrt{2\tau}}&0\\
\frac{1-\tau}{\sqrt{\tau}}&-\tau-\tau\sqrt{2}&0&\frac{\tau-1}{\sqrt{2\tau}}\\
\frac{\sqrt{\tau}(1-\tau)}{\sqrt{2}}&0&-1-\sqrt{2}&\sqrt{\tau}(\tau-1)\\
0&\frac{\sqrt{\tau}(1-\tau)}{\sqrt{2}}&\frac{1-\tau}{\sqrt{\tau}}&-1-\tau\sqrt{2}
\end{pmatrix},\]
\[B_2=\begin{pmatrix}
-1-\tau\sqrt{2}&\frac{\tau-1}{\sqrt{\tau}}&\frac{\sqrt{\tau}(\tau-1)}{\sqrt{2}}&0\\
\sqrt{\tau}(1-\tau)&-1-\sqrt{2}&0&\frac{\sqrt{\tau}(\tau-1)}{\sqrt{2}}\\
\frac{1-\tau}{\sqrt{2\tau}}&0&-\tau-\tau\sqrt{2}&\frac{\tau-1}{\sqrt{\tau}}\\
0&\frac{1-\tau}{\sqrt{2\tau}}&\sqrt{\tau}(1-\tau) &-\tau-\sqrt{2}\\
\end{pmatrix},\]
\[B_3=\begin{pmatrix}-\tau-\tau\sqrt{2}&\frac{\tau-1}{\sqrt{\tau}}&\frac{\tau-1}{\sqrt{2\tau}}&0\\
\sqrt{\tau}(1-\tau)&-\tau-\sqrt{2}&0&\frac{\tau-1}{\sqrt{2\tau}}\\
\frac{\sqrt{\tau}(1-\tau)}{\sqrt{2}}&0&-1-\tau\sqrt{2}&\frac{\tau-1}{\sqrt{\tau}}\\
0&\frac{\sqrt{\tau}(1-\tau)}{\sqrt{2}}&\sqrt{\tau}(1-\tau)&-1-\sqrt{2}
\end{pmatrix} \]
is exponentially stable with respect to every periodic switching law,  but is not globally uniformly asymptotically stable.

\section{Conclusions}

We have shown in Theorem \ref{th:not-periodic} above that the question of R. Shorten \emph{et al} published in \cite{ShWiMaWuKi07} has a negative answer: there exist linear switched systems with four states, in four dimensions, for which constant and periodic switching laws can produce only trajectories which converge exponentially to zero, but such that this linear switched system is not GUAS. We note that by a recently-announced result of J. Bochi and the author \cite{BoMo26}, such counterexamples cannot exist in dimension three or lower, so the counterexample presented here is of the lowest dimension possible. 

During the proof of Theorem \ref{th:not-periodic} we used only those properties of the matrices $A_0$ and $A_1$ which were explicitly stated in Theorem \ref{th:example}, and we also made no significant use of the fact that those matrices were specifically two-dimensional. \footnote{While Lemma \ref{le:tensorbollocks} was stated in the specific context of two-dimensional matrices, it holds in arbitrary dimension $d\geq 1$ by essentially the same proof. We also did not use any of the precise structural features of the periodic switching laws described in Theorem \ref{th:example}\eqref{it:unique} other than the uniqueness of the period.} In particular, if a pair of Metzler matrices $A_0, A_1 \in M_d(\R)$ could be found which also satisfy the conclusions of Theorem \ref{th:example}  then the same arguments would result in a set of four $d^2 \times d^2$ Metzler matrices which define a linear switched system which is Lyapunov stable, is not GUAS, and has the property that all trajectories defined by periodic switching laws converge exponentially to zero. Examples considered in \cite{ChGaMa15,FaMaCh09} suggest, but do not prove, that a pair of Metzler matrices with these properties is likely to exist in dimension three. On the other hand, by applying the lifting arguments of \cite{GuShMa07} to the matrices of Theorem \ref{th:example} we anticipate that a pair of Metzler matrices with the same essential properties could be constructed in sufficiently high dimension. Consequently we conjecture that counterexamples to Question \ref{qu:that-question} with the additional property of being Metzler matrices ought to exist in sufficiently high dimension, and perhaps in particular might exist in dimension nine.

The examples presented in this note leave open the matter of whether or not counterexamples to Question \ref{qu:that-question} can exist which have fewer than four switching states. We see no reason in principle why examples with two or three switching states should not exist, but we are not currently able to resolve this question in either direction. 

\appendix

\section{Proof of Theorem \ref{th:example}}
In this appendix we indicate the proof of Theorem \ref{th:example}. We will make use of two results on Lyapunov norms of linear switched systems which are due respectively to N.E. Barabanov \cite{Ba88b} and Y. Chitour, M. Gaye and P. Mason \cite{ChGaMa15}; the proof is otherwise largely self-contained.

Choose any $A_0,A_1 \in M_2(\R)$ which satisfy
\begin{equation}\label{eq:condition}\max_{\gamma \in [0,1]} \rho\left(e^{(1-\gamma) A_0 + \gamma A_1)}\right) < \sup_{t_0,t_1>0} \rho\left(e^{t_0 A_0}e^{t_1A_1}\right)^{\frac{1}{t_0+t_1}}\end{equation}
where $\rho$ denotes spectral radius. Here the first quantity may be interpreted as the maximal exponential growth rate of a trajectory with constant switching law, and the second the maximal growth rate under periodic bang-bang switching laws with two switching intervals. The reader may easily compute that the example
\[A_0:=\begin{pmatrix}0&1 \\-2&0\end{pmatrix}, \qquad A_1:=\begin{pmatrix}0&2\\-1&0\end{pmatrix}\]
satisfies \eqref{eq:condition} by considering $t_0=t_1=1$, say. 
Now define the maximal growth rate of all trajectories of $A_0,A_1$,
\[\Lambda:=\lim_{t \to \infty}\sup_{x,u} \frac{1}{t}\log \left(\frac{\|x(t)\|}{\|x(0)\|}\right),\]
where the supremum is over all absolutely continuous functions $x \colon [0,\infty) \to \R^2$ which solve \eqref{eq:that-equation} for some measurable switching law $u \colon [0,\infty) \to [0,1]$ and some nonzero initial condition $x(0)$ on the unit circle. The existence of the limit $\Lambda$ is guaranteed by Fekete's lemma. By replacing $A_0$ and $A_1$ with $A_0-\Lambda \cdot I$ and $A_1-\Lambda \cdot I$ if necessary, we will additionally assume without loss of generality that $\Lambda$ is equal to zero. Clearly this subtraction of $\Lambda \cdot I$ does not change the validity of \eqref{eq:condition} since the effect of this transformation is to multiply both sides by precisely $e^{-\Lambda}$. Once this transformation has been made the right-hand side of the inequality is clearly at most $1$ and hence the left-hand side is strictly less than $1$, implying that every $(1-\gamma) A_0+\gamma A_1$ is Hurwitz. It is easily checked that \eqref{eq:condition} is impossible if $A_0$ and $A_1$ have a shared real eigenvector (or equivalently, if there exists an invertible matrix $X$ such that $XA_0X^{-1}$ and $XA_1X^{-1}$ are both upper triangular) and is also impossible if $I$, $A_0$ and $A_1$ are linearly dependent. We have established \eqref{it:hurwitz}.

We now prove \eqref{it:lyapunov} and the clause of \eqref{it:unique} which guarantees the existence of a periodic trajectory. The real number
\[\min_{\|w\|=1} \min_{\gamma \in [0,1]} \left|\det [((1-\gamma) A_0 + \gamma A_1)w,w]\right|\]
cannot be zero since this would imply the existence of a zero  eigenvalue for some $(1-\gamma) A_0 + \gamma A_1$, contradicting the Hurwitz property. Hence there exists $\kappa>0$ such that either 
\begin{equation}\label{eq:kappa1}\min_{\|w\|=1} \min_{\gamma \in [0,1]} \det [((1-\gamma) A_0 + \gamma A_1)w,w] \geq \kappa\end{equation}
or 
\begin{equation}\label{eq:kappa2}\max_{\|w\|=1} \max_{\gamma \in [0,1]} \det [((1-\gamma) A_0 + \gamma A_1)w,w] \leq -\kappa.\end{equation}
The matrices $A_0$ and $A_1$  do not have a shared real eigenvector, so by a theorem of N.E. Barabanov (\cite{Ba88b}, see also \cite[Theorem 4.8]{ChMaSi25}) it follows that there exists a norm $\threebar{\cdot}$ on $\R^2$ which is non-increasing along all trajectories of \eqref{eq:that-equation}, and such that additionally for every $w \in \R^2$ there exists a solution $x(t)$ with initial position $w$ such that $\threebar{x(t)}$ is constant. The first of these two properties immediately yields (ii). Now let $y(t)$ be a solution to \eqref{eq:that-equation} with arbitrary initial condition $y(0):=w \neq 0$ and measurable switching law $v$, say, and along which $\threebar{y(t)}$ is constant and nonzero. The angular velocity of this solution is a.e. equal to $\det [\dot{y}(t), y(t)] / \|y(t)\|^2$ and in view of \eqref{eq:kappa1}--\eqref{eq:kappa2} it follows that this angular velocity has a consistent sign and is uniformly bounded away from zero. Thus $y(t)$ must perform infinitely many rotations around the origin in a consistent direction. In particular there exists $T>0$ such that $y(T)$ is a positive scalar multiple of $y(0)$, and since $\threebar{y(T)}=\threebar{y(0)}$ we must have $y(T)=y(0)$. Defining $u \colon [0,\infty) \to [0,1]$ to be the periodic switching law obtained by periodic repetition of $v|_{[0,T]}$, and similarly defining $x \colon [0,\infty) \to \R^2$ by periodically repeating $y|_{[0,T]}$, it follows that there exists a periodic solution of \eqref{eq:that-equation} with initial condition $w$. This proves the existence statement of \eqref{it:unique}, but it remains to show that the function $u$ is unique and has the required form $u(t)=u_0(t+t_0)$ a.e.

In order to prove this remaining part of \eqref{it:unique} we first collect some facts concerning the unit circle of the norm $\threebar{\cdot}$. Let $\Omega$ denote this unit circle. Since $\Omega$ is the boundary of a convex region, it has at least one point of differentiability; since by the previous observations there exists a periodic trajectory which visits every point of this unit circle, we may apply \cite[Proposition 4.35]{ChMaSi25} to deduce that $\Omega$ is a $C^1$ curve. For $i=0,1$  let $\Omega_i$ denote the set of all $w \in \Omega$ such that $A_iw$ is a tangent direction to $\Omega$ at $w$. Clearly each $\Omega_i$ contains all of its limit points and is radially symmetric, i.e. $w \in \Omega_i$ if and only if $-w \in \Omega_i$. We claim that every $w \in \Omega$ belongs to either $\Omega_0$ or $\Omega_1$. Indeed, if $w$ does not belong to this union then either both of the vectors $A_iw$ point from $w$ into the interior of the unit ball of $\threebar{\cdot}$, or at least one of them points towards the exterior. In the former case all trajectories originating at $w$ would have to enter the open unit ball of $\threebar{\cdot}$, contradicting the existence of a trajectory along which $\threebar{\cdot}$ is constant; in the latter case there would exist a trajectory originating at $w$ along which $\threebar{\cdot}$ locally increases, which is also a contradiction. The claim is proved. We finally observe that if $w \in \Omega_0\cap\Omega_1$ then since $\Omega$ is a differentiable curve the vectors $A_0w$ and $A_1w$ must be proportional to one another, hence $w$ is an eigenvector of $A_0^{-1}A_1$. Since $A_0, A_1$ are linearly independent, $A_0^{-1}A_1$ is not a scalar multiple of the identity and has at most two distinct real eigenvectors. If $\Omega_0 \cap \Omega_1$ is not empty, therefore, it can contain either exactly two or exactly four points, corresponding to the intersection points of $\Omega$ with the real eigenspaces of $A_0^{-1}A_1$ (if any such eigenspaces exist). 
 
We next show that neither $\Omega_0$ nor $\Omega_1$ is equal to the whole of $\Omega$. Suppose for a contradiction that $\Omega=\Omega_i$, say. We know that there exists a periodic trajectory $x(t)$ which lies entirely in $\Omega$, and the tangent to this trajectory is therefore a.e. tangent to $\Omega$, hence in this case $\dot{x}(t)=A_ix(t)$ for a.e. $t$; since $A_i$ is Hurwitz such a trajectory cannot be periodic, and we have obtained a contradiction. This implies that each of $\Omega_0$ and $\Omega_1$ contains an arc of nonzero length, with the endpoint of each arc belonging to $\Omega_0 \cap \Omega_1$. The radial symmetry of these sets makes it impossible for either $\Omega_0$ or $\Omega_1$ to contain just one single arc, so each contains at least two arcs and therefore $\Omega_0 \cap \Omega_1$ must contain at least four points. We already know that $\Omega_0 \cap \Omega_1$ can contain \emph{at most} four points, so $\Omega_0 \cap \Omega_1$  contains exactly four points and each $\Omega_i$ is the union of two closed arcs in $\Omega$, one of which is the image of the other under the symmetry $w \mapsto -w$. 

We may now prove the remainder of \eqref{it:unique}. We will show that if $x(t)$ is a nonzero periodic trajectory generated by a switching law $u$, say, then there exists $t_0 \geq 0$ such that $u(t)=u_0(t+t_0)$ a.e, where $u_0 \colon [0,\infty) \to [0,1]$ is a fixed switching law which satisfies $u_0(t)=0$ for all $t \in [0,T_0)$, $u_0(t)=1$ for all $t \in [T_0,T_0+T_1)$, and $u_0(t)\equiv u_0(t+T_0+T_1)$, for some fixed real numbers $T_0,T_1>0$. 

By homogeneity it suffices to consider only initial points $w \in \Omega$. We begin by finding a single initial point $w_0$ which is the origin of a unique periodic trajectory. Since every trajectory $x(t)$ travels with positive speed in the same angular direction, we may choose $w\in \Omega_0 \cap \Omega_1$ such that every  trajectory originating at $w$ travels immediately into the interior of the arc $\Omega_0$. Fix such a vector $w_0$ and let $x_0(t)$ be a periodic trajectory such that $x_0(0)=w_0$ and $\threebar{x_0(t)}$ is constant. Choose a maximal open interval $(0,T_0)$ such that $x_0(t)$ lies in the interior of $\Omega_0$. Clearly $\dot{x}_0(t)$ must be tangent to $\Omega$ for a.e. $t \in (0,T_0)$, so by the definition of $\Omega_0$ we have $u(t)=0$ a.e. in that interval. Necessarily $x_0(T_0)\in \Omega_0 \cap \Omega_1$ and in the same manner $x_0(t)$ lies in the interior of $\Omega_1$ for all $t$ in some maximal open interval $(T_0,T_0+T_1)$, say, in which case $u(t)=1$ for a.e. $t \in (T_0,T_0+T_1)$ by the same reasoning. We have now traversed one arc of $\Omega_0$ and one arc of $\Omega_1$ and have therefore arrived diametrically opposite our point of origin at $x_0(T_0+T_1)=-w_0$. Using the symmetry of $\Omega_0$ and $\Omega_1$ we obtain $u(t)=0$ a.e. for $t \in (T_0+T_1, 2T_0+T_1)$ and $u(t)=1$ a.e. for $t \in (2T_0+T_1, 2T_0+2T_1)$, and $x_0(2T_0+2T_1)=x_0(0)$ by similar reasoning. Repeating these steps inductively shows that $x_0(t)\equiv x_0(t+2T_0+2T_1)$ and that $u(t)=u(t+T_0+T_1)$ a.e, and thus $u(t)$ is a.e. equal to the function $u_0(t)$ just defined. This proves the uniqueness clause for periodic trajectories originating at the point $w_0\in \Omega$. But every $w \in \Omega$ can be written as $w=x_0(t_0)$ for some $t_0 \in [0,2T_0+2T_1)$, and in particular is the initial point of the trajectory $x(t):=x_0(t+t_0)$ with switching law $u(t):=u_0(t+t_0)$. If there existed a distinct periodic trajectory $y(t)$ originating at this $w$ with a switching law $v(t)$ not a.e. equal to $u_0(t+t_0)$, then since that trajectory would at some future time pass through $w_0$, by alternating between the two distinct switching laws  we could construct a distinct periodic trajectory originating at $w_0$ with a switching law not equal to $u_0(t)$, contradicting the result just shown. Since $u_0(t+t_0)$ obviously has the same period as $u_0(t)$ we have proved the full statement of \eqref{it:unique} with $T:=T_0+T_1$.

It remains to prove \eqref{it:exp}. Given a measurable periodic function $u \colon [0,\infty) \to [0,1]$ with period $\tau$, say, define $R \colon [0,\infty) \to M_2(\R)$ to be the unique absolutely continuous solution to the initial value problem
\[\dot{R}(t)=((1-u(t))A_0+u(t)A_1)R(t),\qquad R(0)=I\]
and note that for every $w \in \R^2$ the function $x(t):=R(t)w$ is precisely the solution to \eqref{eq:that-equation} with initial condition $w$. It follows directly from this observation that $\threebar{R(t)} \leq 1$ for all $t\geq 0$ and in particular $\rho(R(\tau))\leq 1$. Since $u$ is periodic with period $\tau$ the functions $t\mapsto R(t+\tau)$ and $t\mapsto R(t)R(\tau)$ solve an identical  initial value problem and are therefore identical; it follows by induction that $R(n\tau)=R(\tau)^n$ for every $n \geq 0$. By Jacobi's formula the determinant $D(t):=\det R(t)$ satisfies
\[\dot{D}(t)=((1-u(t))\tr A_0 + u(t)\tr A_1 )D(t),\qquad D(0)=1\]
which has solution
\[D(t)=\exp\left(\int_0^t (1-u(s))\tr A_0 + u(s)\tr A_1 ds\right)\]
so in particular $\det R(\tau) \in (0,1)$ using the fact that every $(1-\gamma) A_0+\gamma A_1$ is Hurwitz and thus has negative trace.
If $\rho(R(\tau))<1$ then every solution $x(t)\equiv R(t)x(0)$ satisfies 
\[\limsup_{n \to \infty} \threebar{x(n\tau)}^{1/n}\leq \lim_{n \to \infty} \threebar{R(\tau)^n}^{1/n}
=\rho(R(\tau))<1\] using Gelfand's formula, and  it follows easily that $x(t) \to 0$ exponentially. Otherwise $\rho(R(\tau))=1$, and since $0<\det R(\tau)<1$ exactly one of the eigenvalues of $R(\tau)$ has modulus $1$ and both eigenvalues are real. Thus either $1$ or $-1$ must be an eigenvalue for $R(\tau)$, so $R(2\tau)$ has an eigenvalue $1$ with corresponding eigenvector $w$, say. The trajectory $x(t):=R(t)w=R(t+2\tau)w$ is thus periodic with period $2\tau$ and by \eqref{it:unique} this implies that $u$ has the form $u(t)=u_0(t+t_0)$ a.e. as described in \eqref{it:unique}. The proof is complete.

\bibliographystyle{IEEEtran}
\bibliography{periodiq}
\end{document}